\documentclass[11pt]{amsart}

\usepackage[utf8]{inputenc}

\usepackage{amssymb, graphicx}

\usepackage{amsfonts}
\usepackage{amsmath}
\usepackage{latexsym}
\usepackage{amscd}
\usepackage{verbatim}

\newcommand{\mf}{\mathfrak}
\newcommand{\g}{\mf{g}}
\newcommand{\h}{\mf{h}}

\allowdisplaybreaks[1]

\def\span{\operatorname{span}}

\newcommand{\Z}{{\mathbb Z}}
\newcommand{\C}{{\mathbb C}}
\newcommand{\N}{{\mathbb N}}

\newcommand{\K}{{\mathcal K}}
\newcommand{\M}{{\mathcal M}}

\renewcommand{\a}{\ensuremath{\alpha}}
\renewcommand{\l}{\ensuremath{\lambda}}

\newcommand{\ad}{\operatorname{ad}\xspace}

\renewcommand{\phi}{\varphi}

\renewcommand{\leq}{\leqslant}
\renewcommand{\geq}{\geqslant}

\newcommand{\F}{\mathbb{F}}
\newcommand{\LL }{\mathcal{L}}
\newcommand{\II}{\mathcal{I}}
\newcommand{\bb}{\mathfrak{b}}
\newcommand{\spn}{\mathrm{span}}

\def\sl{\mathfrak{sl}}

\def\l{\lambda}

\def\ad{\text{ad}}
\def\span{\text{span}}

\def\sl{\mathfrak{sl}}
\def\C{\mathbb{C}}
\def\Z{\mathbb{Z}}
\def\N{\mathbb{N}}
\def\bt{{\bf{t}}}
\def\bm{{\bf{m}}}
\def\bn{{\bf{n}}}
\def\l{\lambda}
\def\LL{\mathcal{L}}

\newcommand{\tg}{\tilde {\mathfrak g} }
\newcommand{\hg}{\hat {\mathfrak g} }

\newtheorem{theorem}{Theorem}[section]
\newtheorem{proposition}[theorem]{Proposition}
\newtheorem{lemma}[theorem]{Lemma}
\newtheorem{corollary}[theorem]{Corollary}

\theoremstyle{remark}
\newtheorem{remark}[theorem]{Remark}

\newtheorem{example}[theorem]{Example}
\numberwithin{equation}{section}

\begin{document}
\title[Zero product determined Lie algebras]{Zero product determined Lie algebras}

\author{Matej Brešar,  Xiangqian Guo,  Genqiang Liu, Rencai L\" u, Kaiming Zhao}
\address{M. B.: Faculty of Mathematics and Physics,  University of Ljubljana,  and Faculty of Natural Sciences and Mathematics, University
of Maribor, Slovenia}
\email{matej.bresar@fmf.uni-lj.si}
\address{X. G.: School of Mathematics and Statistics, Zhengzhou University,
Zhengzhou, P. R. China}
\email{guoxq@zzu.edu.cn }
\address{G. L.: School of Mathematics and Statistics, Henan University, Kaifeng, P. R. China}
\email{liugenqiang@amss.ac.cn}
\address{R. L.: Department of Mathematics, Soochow University, Suzhou, P. R. China }
\email{rencail@amss.ac.cn}
\address{K. Z.: Department of Mathematics, Wilfrid Laurier University, Waterloo,  Canada, and   College of Mathematics and Information Science, Hebei Normal  University, Shijiazhuang,  P. R. China }
\email{kzhao@wlu.ca}

\begin{abstract}
A Lie algebra $L$ over a field $\mathbb{F}$ is said to be  zero product determined (zpd) if  every bilinear map
$f:L\times L\to \mathbb{F}$ with the property that $f(x,y)=0$ whenever $x$ and $y$ commute is a coboundary.
The main goal of the paper is to determine whether or not some important Lie algebras are zpd. We show that
the Galilei Lie algebra $\mathfrak{sl}_2\ltimes V$, where $V$ is a simple $\mathfrak{sl}_2$-module, is zpd if and only if $\dim V =2$ or $\dim V$ is odd.
The class of zpd Lie algebras also includes the quantum torus Lie algebras $\mathcal{L}_q$ and $\mathcal{L}^+_q$, the untwisted  affine Lie algebras,
the Heisenberg Lie algebras, and
all Lie algebras of dimension at most $3$, 
while the class of non-zpd Lie algebras includes
the ($4$-dimensional) aging Lie algebra $\mathfrak {age}(1)$ and all Lie algebras of dimension more than $3$ in which only linearly dependent elements  commute.
 We also give some evidence of  the usefulness of the concept of a zpd Lie algebra by using it in the study of commutativity preserving linear maps.
	\end{abstract}


\maketitle

\section{Introduction}
We say that a Lie algebra $L$ over a field $\F$ is {\em  zero product determined} ({\em zpd} for short) if for every bilinear map $f:L\times L\to \F$ with the property
\begin{equation}\label{e1}
[x,y] =0\Longrightarrow f(x,y)=0,\ \forall\ x,y\in L,
\end{equation}
there exists a linear map $\Phi:[L,L]\to \F$ such that
 \begin{equation}\label{e2}
f(x,y)= \Phi\big([x,y]\big),\ \forall\ x,y\in L.
\end{equation}
Thus, borrowing  the Lie algebra cohomology terminology, 
 $L$ is zpd if  \eqref{e1} implies that $f$ is a coboundary.  

The question whether \eqref{e1} implies \eqref{e2} was first studied in \cite{BS}, for the case where $L$ is the Lie algebra of all $n\times n$ matrices. The  answer, which is easily seen to be positive, was used as a crucial tool in describing commutativity preserving linear maps on  matrix algebras. The concept of a zpd algebra was introduced slightly later, in \cite{BGS}, not only for Lie algebras but for general, not necessarily associative algebras (in the definition, one just replaces $[\,\cdot\,, \,\cdot\,]$ in \eqref{e1} and \eqref{e2} by the product in the algebra in question). So far, however, the focus was primarily  on associative algebras -- not only in pure algebra where the concept is quite well understood,  especially in finite dimensions \cite{B}, but also, and in fact  more extensively,  in functional analysis where one additionally assumes that the maps in question are continuous (see \cite{ABEV} and the subsequent series of papers on the so-called Banach algebras with property $\mathbb B$). Not much has been known, so far, about which Lie algebras are zpd. 
This question was studied sporadically in some papers, including the seminal work \cite{BGS}, and more systematically in \cite{Gl} and \cite{WXZ} (and \cite{ABEV2} in the analytic context). 
In  \cite{WXZ} it was shown that parabolic  subalgebras of complex finite dimensional simple Lie algebras are zpd. This readily implies 
that complex finite dimensional semisimple  Lie algebras are zpd, to point out one basic example. For infinite dimensional Lie algebras, as well as  for other finite dimensional Lie algebras, the question is largely open. The aim of
this paper is to remedy this situation. We will prove that various different Lie algebras are zpd, but also  various different are not. 
Given a Lie algebra, it is usually not easy to 
guess  whether it is  zpd or not. 
This makes the problem challenging and, in our opinion, interesting in its own right. Another motivation for us for studying it is the striking similarity between the notion of a zpd Lie algebra and the notion of a group with  trivial Bogomolov multiplier  (see, e.g.,~\cite{M}).
Finally, as we will see at the end of the paper, knowing that a Lie algebra is zpd can be used for describing commutativity preserving linear maps.  Hopefully other applications will be found later, just as they  were found for zpd associative algebras (cf. \cite{ABEV, B}).

The paper is organized as follows.  In Section \ref{s2} we survey some general facts on zpd Lie algebras, and also introduce and study zero product determined modules which play an important role in some of the subsequent sections. Sections \ref{s3}-\ref{s5} are rather easy;
in Section \ref{s3} we show that every Lie algebra of dimension at most 3 is zpd, in Section \ref{s4} we show that a Lie algebra  in which only  linearly dependent commute is not zpd if its dimension is greater than $3$, and in Section \ref{s5} we show that the Heisenberg Lie algebras  are zpd. The next sections contain deeper results. 
In Section \ref{s6} we show that the Galilei Lie algebra $\sl_2\ltimes V$, where $V$ is a simple $\sl_2$-module, is zpd if and only if $\dim V =2$ or $\dim V$ is odd.
We also find a $4$-dimensional algebra, namely  the aging algebra $\mathfrak{age}(1)$, that is not zpd. In Section \ref{s7} we  show that the quantum torus Lie algebras $\LL_q$ and $\LL^+_q$ are zpd, and in Section \ref{s8} we show that the untwisted affine Lie algebras are zpd. The final section, Section \ref{s9}, is devoted to applications of the zpd
property to  commutativity preserving linear maps.

\section{General remarks on zpd Lie algebras}\label{s2}

In this section we gather together several elementary observations on zpd Lie algebras, some new and some already known.
First a word on notation: throughout, $\F$ denotes a field, and all algebras, modules and vector spaces are assumed to be over $\F$. Sometimes we will assume that $\F=\C$; if this is not assumed, then $\F$ is just an arbitrary field. By $\N$ and  $\Z_+$ we denote the set of all positive integers and all nonnegative integers.

We begin by giving some equivalent
 definitions of zpd Lie algebras (they are actually valid in greater generality, but we will formulate them for Lie algebras for convenience).
First of all, note that the existence of a linear map  $\Phi:[L,L]\to \F$ satisfying \eqref{e2} is equivalent to the condition
\begin{equation}\label{e3}
\sum_i [x_i,y_i]=0\Longrightarrow \sum_i f(x_i,y_i)=0,\ \forall\ x_i,y_i\in L.
\end{equation}
Indeed, if \eqref{e3} holds, then the map $\Phi:[L,L]\to X$ given by $$\Phi\Bigl(\sum_i [x_i,y_i]\Bigr)= \sum_i f(x_i,y_i)$$ is well-defined and satisfies \eqref{e2}, and the converse is obvious.
 In the next lemma we show that  bilinear maps from the definition may have their ranges in arbitrary vector spaces, not necessarily in the field of scalars
(this is actually taken as the definition in the seminal paper \cite{BGS}).
 The proof is very easy, but the fact itself is important for applications.

\begin{lemma}\label{lx}
Let $L$ be a zpd Lie algebra. If $X$ is an arbitrary vector space  and $f:L\times L\to X$ is a bilinear map satisfying
\begin{equation*}\label{zpd}
[x,y] =0\Longrightarrow f(x,y)=0,\ \forall\ x,y\in L,
\end{equation*}
then
there exists a linear map $\Phi:[L,L]\to X$ such that
$$
f(x,y)= \Phi\big([x,y]\big),\ \forall\ x,y\in L.
$$
\end{lemma}

\begin{proof}
Take a linear functional $\xi$ on $X$. The map $f_\xi:=\xi\circ f:L\times L\to \F$ satisfies the condition that $[x,y]=0$ implies $f_\xi(x,y)=0$. Consequently,
there exists a linear functional $\Phi_\xi$ on $[L,L]$ such that $f_\xi(x,y)=\Phi_\xi([x,y])$ for all $x,y\in L$. Thus, for all $x_i,y_i\in L$, $\sum_i [x_i,y_i]=0$ implies
$\sum_i f_\xi (x_i,y_i)=0$, i.e., $\xi\big(\sum_i f (x_i,y_i)\big)=0$. Since $\xi$ is an arbitrary linear functional on $X$, this actually shows that
$f$ satisfies \eqref{e3}.
 \end{proof}

Given a Lie algebra $L$, denote by $\M_L$ the kernel of the linear map $L\otimes L\to L$ defined by $x\otimes y\mapsto [x,y]$, i.e.,
$$\M_L= \Big\{\sum_i x_i\otimes y_i\in L\otimes L \,|\, x_i,y_i\in L,\,\,\sum_i [x_i,y_i]=0\Big\}.$$
If $L$ is finite dimensional, we see that $\dim\M_L=(\dim L)^2-\dim [L,L]$.
Further, let us write
$$\K_L= \{x\otimes y\in L\otimes L\,|\,x,y\in L,\, [x,y]=0\}.$$
Obviously, we have
$${\rm span}\, \K_L\subseteq \M_L.$$
We also remark that
\begin{equation}\label{xyyx}
x\otimes y + y\otimes x = (x+y)\otimes (x+y) - x\otimes x - y\otimes y \in {\rm span}\,\K_L,\ \forall\ x,y\in L.
\end{equation}
The next lemma is essentially \cite[Theorem 2.3]{BH}. Nevertheless, we give the proof for the sake of completness.

\begin{lemma}\label{lemkm}
A Lie algebra $L$ is zpd if and only if ${\rm span}\, \K_L=\M_L$.
\end{lemma}

\begin{proof}
Suppose that $ {\rm span}\, \K_L\ne \M_L$. Then there exists a linear functional $\xi$ on $L\otimes L$ such that
$\xi(\K_L) =\{0\}$ and
$\xi(\M_L) \ne\{0\}$. The bilinear map $f:L\times L\to \F$ given by $f(x,y)=\xi(x\otimes y)$, therefore satisfies \eqref{e1} but not \eqref{e3}. To prove the converse,
assume
 that $L$ is not zpd. Then there exists a bilinear map $f:L\times L\to \F$ that satisfies \eqref{e1} but not \eqref{e3}. The linear functional
$\xi:L\otimes L\to \F$, $\xi(x\otimes y)=f(x,y)$, therefore vanishes on $\K_L$ but not on $\M_L$. Accordingly, ${\rm span}\, \K_L\neq\M_L$.
\end{proof}

It is even more convenient to consider the wedge product $L\wedge L$ and define
$$\M'_L= \Big\{\sum_i x_i\wedge y_i\in L\wedge L \,|\, x_i,y_i\in L,\,\,\sum_i [x_i,y_i]=0\Big\},$$
$$\K'_L= \{x\wedge y\in L\wedge L\,|\,x,y\in L,\, [x,y]=0\}.$$
The following results are obvious.

\begin{lemma}\label{Wedge} Let $L$ be a Lie algebra.
\begin{enumerate}
\item[(a)]
 $L$ is zpd if and only if ${\rm span}\, \K'_L=\M'_L$.

\item[(b)] If the linear map $L\wedge L\to [L, L], x\wedge y\mapsto [x, y]$ is a linear isomorphism, then $L$ is zpd.

\item[(c)] If $\dim\,L\wedge L=\dim\, [L, L]< \infty$, then $L$ is zpd.

\item[(d)] If $\dim\,(L\wedge L)/{\rm span}\,\{x\wedge y | x,y\in L, [x, y]=0\}=\dim\, [L, L]< \infty$, then $L$ is zpd.

\end{enumerate}
\end{lemma}

\begin{remark}
The notion of a zpd Lie algebra may remind one on the notion of a {\em centrally closed Lie algebra}, i.e., a Lie algebra on which every $2$-cocycle is a coboundary (equivalently, the second cohomology group $H^2(L,\F)$ is trivial). Let us therefore make some comment on this. Recall that a $2$-cocycle on a Lie algebra $L$ is a skew-symmetric bilinear map
$f:L\times L\to\F$ satisfying
$$
f([x,y],z) + f([z,x],y) + f([y,z],x)=0,\ \forall\ x,y,z\in L,
$$
and that a coboundary is a bilinear map $f:L\times L\to\F$ of the form \eqref{e2} for some linear map $\Phi:[L,L]\to \F$.
A coboundary is obviously a $2$-cocycle, while the converse is not always true. Setting
$$\mathcal C_L= \{x\otimes y + y\otimes x, [x,y]\otimes z + [z,x]\otimes y + [y,z]\otimes x\,|\,x,y,z\in L\}$$
one can show, just by following the proof of Lemma \ref{lemkm}, that
$$\mbox{$L$ is centrally closed}\iff {\rm span}\, \mathcal C_L=\M_L.$$
The spaces  ${\rm span}\, \mathcal C_L$ and ${\rm span}\, \K_L$ do not seem to be connected in some simple way, and the problem of showing that $L$ is zpd is,  in general, independent of the problem of showing that $L$ is centrally closed.  Given any Lie algebra $L$, it may be of  interest to describe and compare the spaces
 ${\rm span}\, \mathcal C_L$, ${\rm span}\, \K_L$, and $\M_L$. However, in this paper we will not touch upon this problem.
\end{remark}

The next lemma provides two conditions under which a homomorphic image of a zpd Lie algebra is again zpd. The first condition, (a), is  known \cite[Corollary 5.3]{BH}, but we give the proof anyway.

\begin{lemma} \label{lab}
Let $L$ be a zpd Lie algebra. If an ideal $I$ of $L$ satisfies at least one of the following two conditions:
\begin{enumerate}
\item[(a)] $[L,I] = [L,L]\cap I$,
\item[(b)] the Lie algebra $L/I$ is centrally closed,
\end{enumerate}
then $L/I$ is also a zpd Lie algebra.
\end{lemma}

\begin{proof}
Let $f:L/I\times L/I\to F$ be a bilinear map such that $f(r,s)=0$ whenever $r,s\in L/I$ commute. Then $\bar{f}:L\times L\to F$,
$\bar{f}(x,y)=f(x+I,y+I)$ satisfies $\bar{f}(x,y) =0$ whenever $x,y\in L$ commute. Since $L$ is zpd it follows that
\begin{equation}\bar{f}(x,y) = \bar{\Phi}([x,y])\label{a}\end{equation}
 for some linear map $\bar{\Phi}:[L,L]\to F$.

Suppose that (a) holds. We claim that the map $\Phi:[L/I,L/I]\to F$ given by   $$\Phi\Big(\sum_i [x_i+I,y_i+I]\Big) = \sum_i \bar{\Phi}\big([x_i,y_i]\big)$$
is  well-defined. Indeed, suppose that $\sum_i [x_i,y_i]\in I$. Since this element belongs also to $[L,L]$, (a) implies that
$$\sum_i [x_i,y_i] = \sum_j [z_j,u_j]$$ for some $z_j\in L$ and $u_j\in I$. From
$$\bar{\Phi}([z_j,u_j])= \bar{f}(z_j,u_j)=f(z_j + I, u_j + I) = f(z_j + I, 0)=0$$
we infer that
$$\sum_i \bar{\Phi}\big([x_i,y_i]\big) = \bar{\Phi}\Big(\sum_i [x_i,y_i]\Big) =  \bar{\Phi}\Big(\sum_j [z_j,u_j]\Big) = \sum_j \bar{\Phi}\big([z_j,u_j]\big)=0,$$
as desired. We therefore have
$$f(x+I,y+I) = \bar{f}(x,y)=\bar{\Phi}([x,y])= \Phi([x+I,y+I]),\ \forall\ x,y\in L,$$
proving that $L/I$ is zpd.

Assume now that (b) holds. Using \eqref{a} it follows immediately that
$$
\bar{f}([x,y],z) + \bar{f}([z,x],y) + \bar{f}([y,z],x)=0,\ \forall\ x,y,z\in L,
$$
that is,
$$f([r,s],t)+ f([t,r],s)+ f([s,t],r)=0,\ \forall\ r,s,t\in L/I.$$
Since (b) holds, the desired conclusion that $f$ is a coboundary follows.
\end{proof}

We now extend the notion of a zpd Lie algebra to modules as follows.
We will say that a module $V$ over a Lie algebra $L$ is {\em  zero action determined} ({\em  zad} for short)  if for every bilinear map
$f: L \times V\to \F$ with the property that
\begin{equation}\label{zad-1}
xv =0\Longrightarrow f(x,v)=0,\ \forall\ x\in L , v\in V,
\end{equation}
there exists a linear map $\Phi: LV\to \F$ such that
 \begin{equation}\label{zad-2}
f(x,v)= \Phi(xv),\ \forall\ x\in L , v\in V.
\end{equation}
Note that $L$ is a zpd Lie algebra if and only if  its adjoint module is zad.

The notion of a zad module is perhaps interesting in its own right. In this paper, however, we will use it only as a tool for answering the question whether certain Lie algebras
are zpd. 

Until the rest of this section we will consider some basic properties of zad modules. First we record the obvious analogue of Lemma \ref{lemkm}. Setting
$$\M_V= \Big\{\sum_i x_i\otimes v_i\,|\, x_i\in L,v_i\in V,\,\,\sum_i x_iv_i=0\Big\}$$
and
$$\K_V= \{x\otimes v\,|\,x\in L,v\in V,\,\, xv=0\},$$
we have

\begin{lemma}\label{lemkm2}  A module $V$ over a Lie algebra $L$ is zad if and only if ${\rm span}\, \K_V=\M_V$.
\end{lemma}

Let us only mention that  Lemmas \ref{lx} and \ref{lab}\,(a)
can also be easily extended to zad modules, but we do not need these results in this paper. The following elementary lemma, however, is needed.
We leave the proof  to the reader (cf. \cite[Theorem 3.1]{BH} which states that  the direct sum of Lie algebras is zpd if and only if every summand is zpd).


\begin{lemma}\label{lemsum} Let $L$ be a Lie algebra and let
 $\{V_i\,|\,i\in I\}$ be a family of $L$-modules. Then their direct sum $\oplus_{i\in I}V_i$ is zad if and only if each $V_i$, $i\in I$, is zad.
\end{lemma}

As an application of this lemma, we will now show that the zpd property is stable under tensoring with commutative associative  unital algebras.

\begin{lemma}\label{current}Let $L$ be a zpd Lie algebra and let  $A$ be a commutative associative  unital algebra. Then the Lie algebra $L\otimes A$ is zpd. \end{lemma}

\begin{proof}From $$(x\otimes 1+y\otimes a_2)\otimes (x\otimes a_1+ y\otimes a_1a_2)\in \mathcal{K}_{L\otimes A},\ \forall\,\, x,y\in L, \forall\,\, a_1,a_2\in A,$$ it  easily follows that
$$(x\otimes a_1)\otimes (y\otimes a_2) \equiv (x\otimes 1)\otimes (y\otimes a_1a_2)\,\,\big({\rm mod} \,\,\span\, \mathcal{K}_{L\otimes A}\big).$$
 Now for every $\sum_i(x_i\otimes a_i)\otimes (y_i\otimes b_i)\in \mathcal{M}_{L\otimes A}$, we have
$$\sum_i(x_i\otimes a_i)\otimes (y_i\otimes b_i)\equiv \sum_i (x_i\otimes 1)\otimes (y_i\otimes a_ib_i)\,\,\big({\rm mod} \,\,\span\, \mathcal{K}_{L\otimes A}\big).$$
 Since $L$ is zpd, Lemma \ref{lemsum} implies that $L\otimes A$ is a zad module over  $L=L\otimes 1$. Consequently,
 $$\sum_i (x_i\otimes 1)\otimes (y_i\otimes a_ib_i)=\sum_j(z_j\otimes 1)\otimes u_j$$
 for some $z_j\in L, u_j\in L\otimes A$ with $$[z_j\otimes 1,u_j]=0,\ \forall j.$$
We have thus proved that
$$\sum_i(x_i\otimes a_i)\otimes (y_i\otimes b_i)\equiv 0\,\,\big({\rm mod} \,\,\span\, \mathcal{K}_{L\otimes A}\big),$$ showing that $L\otimes A$ is a zpd Lie algebra. \end{proof}

Recall that the {\it semidirect product} of a Lie algebra $L$ and an $L$-module $V$ is the Lie algebra $L \ltimes V$ defined on the space $L \oplus V$ by
$$[x+u, y+v]=[x,y]+xv-yu, \,\,\,\forall\,\,\, x,y\in L, \forall\,\, u,v\in V.$$

\begin{lemma} \label{lsd}
Let $L$ be a zpd Lie algebra and $V$ be a zad $L$-module. Then $L\ltimes V$ is a zpd Lie algebra. \end{lemma}

\begin{proof} Take an element $$\sum\limits_{i} (x_i+u_i)\otimes (y_i+v_i)\in
(L\ltimes V)\otimes(L\ltimes V)$$  such that $$\sum\limits_i [x_i+u_i,y_i+v_i] =0.$$
Our goal is to show that this element lies in  ${\rm span}\,\mathcal{K}_{L\ltimes V}$.
We have $$\sum\limits_i ([x_i, y_i] +x_iv_i-y_iu_i ) =0,$$ which yields  $$\sum\limits_i [x_i, y_i]=0\quad\mbox{and}\quad \sum\limits_i (x_iv_i-y_iu_i )=0.$$ Since $L$ is a zpd Lie algebra and $V$ is a zad $L$-module, we have
 $$\sum\limits_i x_i\otimes y_i\in {\rm span}\,\mathcal{K}_L$$
and
 $$\sum\limits_i (x_i\otimes v_i-y_i\otimes u_i )\in{\rm span}\, \mathcal{K}_V.$$
 Clearly, $\mathcal{K}_L, \mathcal{K}_V\subseteq \mathcal{K}_{L\ltimes V}$. Since, by \eqref{xyyx}, $y_i\otimes u_i+ u_i\otimes y_i\in {\rm span}\,\mathcal{K}_{L\ltimes V}$,
and trivially $ u_i\otimes v_i\in \mathcal{K}_{L\ltimes V}$,
 it follows
 from $$\sum\limits_i (x_i+u_i)\otimes (y_i+v_i)=\sum\limits_i x_i\otimes y_i+ \sum\limits_i (x_i\otimes v_i-y_i\otimes u_i )+\sum\limits_i (y_i\otimes u_i+ u_i\otimes y_i)+\sum\limits_i  u_i\otimes v_i$$
that $\sum\limits_i (x_i+u_i)\otimes (y_i+v_i)\in {\rm span}\,\mathcal{K}_{L\ltimes V}$.
\end{proof}

The converse of the lemma is not true in general.

\begin{example}\label{exx}Let $\bb$ be the noncommutative Lie algebra with a basis $\{h,e\}$ and multiplication given by $[h,e]=e$. Let $A$ be any   commutative  associative unital algebra. Then the Lie algebra $\bb \otimes A=(h\otimes A)\ltimes (e\otimes A)$ is always zpd by Lemma \ref{current}. However, $e\otimes A$ is a zad $h\otimes A$-module if and only if $A$ is zpd as an associative algebra.
\end{example}

\begin{remark}
The definition of a zad module obviously makes sense for associative algebras, too. If $A$ is a unital associative zpd algebra, then every unital module $V$ over $A$ is automatically zad. Indeed, if $f:A\times V\to \F$ is a bilinear map satisfying \eqref{zad-1}, then for any fixed $v\in V$ the map $f_v:A\times A\to \F$, $f_v(x,y)=f(x,yv)$, satisfies the condition that $xy=0$ implies $f_v(x,y)=0$. Since $A$ is zpd it follows that $f_v(x,1)= f_v(1,x)$ for every $x\in A$. This means that $f(x,v) = f(1,xv)$ for all
$x\in A$, $v\in V$, showing that $V$ is zad. Example \ref{exx} shows that a module over a zpd Lie algebra may not be zad.
\end{remark}

Let us show that the converse of  Lemma \ref{lsd} does hold in the special case where $L$ has the property that elements in $L$ commute only when they are linearly dependent. A simple example is the Lie algebra $\sl_2$.

\begin{lemma}\label{semidirect}Let $L$ be a Lie algebra such that  any two commuting elements in $L$ are linearly dependent.
If  $V$ is an $L$-module,
then the Lie algebra $L \ltimes V$ is zpd if and only if $L $ is zpd and $V$ is zad.
\end{lemma}

\begin{proof}
The sufficiency follows from Lemma \ref{lsd}. Now suppose $L \ltimes V$ is zpd.  From Lemma \ref{lab}\,(a) we see that  $L =(L \ltimes V)/V$ is  zpd.
Next we show that $V$ is a zad  $L $-module.

Take any bilinear map $f: L \times V\to \F$ with property \eqref{zad-1}. We can extend $f$ to a bilinear
map $F: (L \ltimes V)\times (L \ltimes V)\to \F$ as follows
$$F(x+u, y+v)=f(x,v)-f(y,u),\ \forall\ x,y\in L  , u,v\in V.$$
First we show that $F$ also satisfies  \eqref{zad-1}.

Take any $x,y\in L  , u,v\in V$ with $[x+u, y+v]=0$.
If $x=0$, then we have $$[x+u, y+v]=-yu=0.$$ Consequently, $f(y,u)=0$, which implies $F(x+u,y+v)=0$ in this case.
Now suppose $x\neq0$. Then $[x+u, y+v]=[x,y]+xv-yu=0$ implies $[x,y]=0$ and $xv-yu=0$.
By the assumption, there exists $a\in\F$ such that $y=ax$. Consequently, $xv-yu=x(v-au)=0$.
Hence $F(x+u,y+v)=f(x,v)-f(y,u)=f(x,v-au)=0$ by the assumption on $f$.

Since we have assume that $L \ltimes V$ is zpd, there is a linear map $\Psi: L \ltimes V\rightarrow \F$
such that $$F(x+u,y+v)=\Psi([x+u,y+v]),\,\,\, \forall \,\,\, x,y\in L  , u,v\in V.$$
In particular, $$f(x,v)=F(x+0,0+v)=\Psi(xv), \,\,\, \forall \,\,\, x\in L  , v\in V.$$ Let $\Phi$ be the restriction of $\Psi$ on $LV$.
Then \eqref{zad-2} holds and $V$ is a zad  $L $-module.
\end{proof}

\section{Lie algebras of low dimensions}\label{s3}

In this section we will prove the following

\begin{proposition} All Lie algebras over any field $\F$ of dimension $\le 3$ are zpd.
\end{proposition}

\begin{proof}  From Lemma \ref{Wedge} we see that all one or two dimensional Lie algebras are zpd. Now suppose $L$ is a
Lie algebra over any field $\F$ of dimension $3$. If it is a direct sum of two nonzero ideals which have to be
one or two dimensional Lie algebras, it follows that $L$  is zpd. Now we assume that $L$ cannot be a direct sum of two nonzero ideals.

{\it Case 1}: $[L, L]=L$.

It is easy to see that $\dim L\wedge L=3$. Then  the linear map $$L\wedge L\to [L,L],\,\,\, x\wedge y\mapsto [x,y],$$ is a linear isomorphism.  From Lemma \ref{Wedge} (b) we know that
$L$ is zpd in this case.

{\it Case 2}: $\dim [L, L]=2$.

Let $[L,L]=\F x\oplus\F y$ and $x,y,z$ form a basis for $L$. If $[x,y]=0$  it follows that
$$\dim\Big((L\wedge L)/\span\,\{u\wedge v | u, v\in L, [u,v]=0\}\Big )=\dim [L,L].$$
 From Lemma \ref{Wedge} (d) we see that $L$ is zpd. Next we may assume that $[x,y]=y$.

 If $[z, \F x+\F y]= \F x$, there is a nonzero $cx+dy$ such that $[z, cx+dy]=0$.

 If  $[z, \F x+\F y]\ne \F x$,
 then $[z,ax+by]=y$ for some $a, b\in\F$. If $b\ne 0$ we know that $[z-b^{-1}x, ax+by]=0$, if $b=0$  it follows that $[z+ax, y]=0$.

For all cases we have
$$\dim\Big((L\wedge L)/\span\,\{u\wedge v | u, v\in L, [u,v]=0\}\Big )=\dim [L,L].$$
 From Lemma \ref{Wedge} (d) we  see that $L$ is zpd.

{\it Case 3}: $\dim [L, L]=1$.

Let $[L,L]=\F w$ and $u, v, w$ form a basis for $L$. We may assume that
$$[u, v]=aw,\,\,\,[u,w]=bw,\,\,\,[v,w]=cw,$$ for some $a,b,c\in\F$.

If $b=c=0$ we see that $L$ is the $3$-dimensional Heisenberg algebra which is zpd. This can be easily checked; on the other hand, we will prove this in Section \ref{s5}.

Now we assume that $b\ne0$ or $c\ne0$. By symmetry we may assume that $b\ne0$. By re-choosing $v$  we have
$$[u, v]=0,\,\,\,[u,w]=bw,\,\,\,[v,w]=cw,$$
By re-choosing $v$  again, we have
$$[u, v]=0,\,\,\,[u,w]=bw,\,\,\,[v,w]=0.$$
We see that $L$ is a direct sum of two ideals $\F v$ and span$\{u, w\}$, which contradicts the assumption on $L$.
Thus $L$ is zpd in this case. This completes the proof.
\end{proof}

From this proposition we know that the lowest dimension of a non-zpd Lie algebra is possibly $4$. In Section 6 we will construct such examples.

\section{Lie algebras in which only linearly dependent elements commute}\label{s4}

The purpose of this section is to describe  zpd Lie algebras $L$ such that any two commuting elements in $L$ are linearly dependent. We first need a technical lemma:

\begin{lemma}\label{subalg-proportional}Suppose $L $ is a zpd Lie algebra  such that  any two commuting elements in $L$ are linearly dependent. Then any subspace of $L $ with dimension $\geq3$ is a subalgebra of $L $.
\end{lemma}

\begin{proof} Since $L $ is zpd,
we see that every element in $L\otimes L$ of the form
$$[x,y]\otimes z + [z,x]\otimes y + [y,z]\otimes x$$
can be written as a linear combination of elements of the form $u\otimes u$, $u\in L$. Since  elements of the latter  form are fixed points of the linear transformation of $L\otimes L$: $r\otimes s\mapsto s\otimes r$, the same should be true for elements of the former form. Accordingly,
$$[x,y]\otimes z + [z,x]\otimes y + [y,z]\otimes x = z\otimes [x,y] + y\otimes [z,x] + x\otimes [y,z]$$
for all $x,y,z\in L$. 
If $x,y,z$ are linearly independent,  it follows that $[x,y]$ lies in the linear span of $x,y,z$ (see, e.g., \cite[Lemma 4.11]{INCA}) 
and hence the space spanned by $x,y,z$ is a subalgebra of $L $.

Now take any subspace $L '\subseteq L $ with $\dim L '\geq3$. For any $x,y\in L  '$ such that
$[x,y]\neq0$,  there exists $z\in L  '$ such that $x,y,z$ are linearly independent. By the previous
argument, we see that the subspace spanned by $x,y,z$ is a subalgebra of $L $. Hence $[x,y]\in\spn\{x,y,z\}\subseteq L '$, which is a subalgebra of $L $ as desired.
\end{proof} 

\begin{theorem}\label{proportionalZPD}
Let $L $ be a zpd Lie algebra such that any two commuting elements in $L$ are linearly dependent. Then $L $ is isomorphic to  a
$3$-dimensional simple Lie algebra, or the $2$-dimensional noncommutative Lie algebra, or
the $1$-dimensional Lie algebra.
\end{theorem}

\begin{proof}
 By the assumption on $L $, we see that any bilinear map $f$ on $L $ satisfying  the property \eqref{e1} is equivalent to a linear map on $L \wedge L $. Hence $L $ is zpd if and only if   the canonical
map $$\pi: L \wedge L \rightarrow [L ,L ],\,\,\, x\wedge y\mapsto [x,y],$$
is a vector space isomorphism.

If $L $ is finite dimensional, we have $\dim L \wedge L =\dim[L ,L ]$, which holds only if $\dim L \leq3$.

If $\dim L =2$, since $L$ cannot be commutative,  it has to be the noncommutative $2$-dimensional Lie algebra.

Now suppose $\dim L =3$. Then $\dim[L ,L ]=\dim L \wedge L =3$. Suppose $L$ is not simple and let $\II$ be a nonzero proper ideal of $L $.
If $\dim\II=1$, there exists $x,y\in L  $ such that $L =\II\oplus\F x\oplus\F y$, which implies
$[L ,L ]\subseteq\spn\{[x,y],\II\}$ and $\dim[L ,L ]\leq2$, contradiction.
If $\dim\II=2$, there exists $x\in L  $ such that $L =\II\oplus\F x$, which implies $[L ,L ]\subseteq\II$
and $\dim[L ,L ]\leq\dim\II$, contradiction. So $L$ is a $3$-dimensional simple Lie algebra.

Next suppose that $L $ is infinite dimensional. 
By Lemma \ref{subalg-proportional}, there is a $4$-dimensional subalgebra $L'\subset L$. We will prove that this is impossible.

Clearly, any two commutative elements of $L '\subseteq L $ are also linearly dependent. For any linearly independent elements $x,y, u, v\in L'$,
since $\{x,y,u\}$ spans a $3$-dimensional Lie algebra, and the same is true for $\{x,y,v\}$, we may assume that
$$[x,y]=a_1x+b_1y+c_1u, \,\,\, [x,y]=a_2x+b_2y+c_2v,$$
for some $a_i, b_i,c_i\in \F$. We deduce that $c_1=c_2=0$. Thus, any two elements in $L$  span a subalgebra.

Since $\{x,y\}$ can span a $2$-dimensional Lie algebra, and  $\{x,u\}$  and $\{x,v\}$ can also, we may assume that
$$[x,y]=ax+by, \,\,\, [x,u]=a'x+b'u, \,\,\, [x,v]=a''x+b''v.$$
Considering the subalgebras spanned by  $\{x,y+u\}$  and $\{x,u+v\}$, we deduce that $b=b'=b''\ne0$.
By rechoosing $x, y$ and  $u$ we may assume that
$$[x,y]=y, \,\,\, [x,u]=u.$$

Since $x,y,u$ form a basis of the Lie subalgebra $L''=$span$\{x,y,u\}$, we see that the only eigenvalues of ad$(x)$ on $L''$ are $0$ and $1$.
From
$$[x,[y,u]]=[[x,y],u]+[y, [x, u]]=2[y,u],$$
we obtain that $[y,u]=0$, or $[y,u]\in \F x$ if char$(\F)=2$,    contradicting the assumptions on $L$.
This completes the proof.
\end{proof}

\begin{corollary}
An infinite dimensional Lie algebra in which  any two commuting elements are linearly dependent is not zpd.
\end{corollary}

Let us point out two special cases of the corollary.

\begin{corollary}
A free Lie algebra in at least two variables is not zpd.
\end{corollary}

\begin{corollary}
The first Witt algebra $W_1$ is not zpd. Moreover, the same is true for every infinite dimensional subalgebra of $W_1$.
\end{corollary}

The fact that $W_1$ is not zpd was already noticed in \cite{Gl}.

\section{Heisenberg Lie algebras}\label{s5}

Recall that a Heisenberg algebra $\h$ is a Lie algebra with basis $\{c, x_i, x_{-i}\,|\,i\in I\}$ where $I$ is a nonempty finite or infinite set, such that
$$[x_i, x_{-i}]=c, \forall i\in I, $$ and all other brackets of basis elements are $0$. See \cite{KR}.

In this section we will prove the following

\begin{proposition} All Heisenberg algebras over any field $\F$  are zpd.
\end{proposition}

\begin{proof}  Let us also write $L$ for the above defined Heisenberg algebra $\h$.
It is clear that $\span\, \K'_L\subseteq \M'_L$, so it is enough to show that $\M'_L\subseteq\span \,\K'_L$.

By $u\cong v$ we mean $u-v\in \span \,\K'_L$. We see that
$$\aligned &c\wedge \h\cong 0,\\ &x_{\pm i}\wedge x_{\pm j}\cong 0, \forall i\ne j\in I,\\
&x_i\wedge x_{-i}\cong x_j\wedge x_{-j}, \forall i, j\in I,\endaligned$$
since $[x_i+x_j, x_{-i}-x_{-j}]=0$.

Take an arbitrary $Y=\sum_{j=1}^my_j\wedge y_j'\in \M'_L$. We may assume that
$$\aligned  &y_j= \sum_{i\in P} a_{ji}x_i+a_{j,-i}x_{-i},\\ &y_j'= \sum_{i\in P}a'_{ji}x_i+a'_{j,-i}x_{-i}, \endaligned$$
where $P$ is a finite subset of $I$, and $a_{j,\pm i}, a'_{j,\pm i}\in \F$. Then
$$\sum_{j=1}^m\sum_{i\in P}(a_{ji}a'_{j,-i}-a'_{ji}a_{j,-i})=0.$$
We see that
$$\aligned Y&=\sum_{j=1}^m\sum_{i\in P}(a_{ji}a'_{j,-i}-a'_{ji}a_{j,-i})x_i\wedge x_{-i}\\
&\cong\sum_{j=1}^m\sum_{i\in P}(a_{ji}a'_{j,-i}-a'_{ji}a_{j,-i})x_1\wedge x_{-1}=0.
\endaligned$$
Thus $\h$ is zpd in this case. This completes the proof.
\end{proof}

\section{Galilei Lie algebras}\label{s6}

Let $\sl_2$ be the $3$-dimensional simple Lie algebra over $\C$, i.e., the Lie algebra consisting of all traceless $2\times 2$ matrices.
Let $$E=\left(
                       \begin{array}{cc}
                         0 & 1 \\
                         0 & 0 \\
                       \end{array}
                     \right), H=\left(
                       \begin{array}{cc}
                         1 & 0 \\
                         0 & -1 \\
                       \end{array}
                     \right), F=\left(
                       \begin{array}{cc}
                         0 & 0 \\
                         1 & 0 \\
                       \end{array}
                     \right)$$
be the standard basis of $\sl_2(\C)$.
Denote by  $\ad$  the adjoint representation of $\sl_2$, i.e., $(\ad A)(B)=[A,B]$ for all $A,B\in\sl_2$. An element $A\in\sl_2$ is called
$\ad$-nilpotent if $\ad A$ is nilpotent; similarly, $A$ is called $\ad$-semisimple if $\sl_2$ can be
decomposed as the sum of engeispaces with respect to the action of $\ad A$.

For any nonzero 
$A\in \sl_2(\C)$, it is clear that $A$ has eigenvalues $\pm \sqrt{-\det(A)}$.
Thus $A$ is ad-nilpotent if $\det(A)=0$ and ad-semisimple if $\det(A)\ne 0$. In particular, if $A\in \sl_2(\C)$ is ad-semisimple, then there exists $P\in$GL$_2$ such that $P^{-1}AP=\det(A)H$; if $A\in \sl_2(\C)$ is ad-nilpotent and nonzero, then there exists $P\in$GL$_2$ such that $P^{-1}AP=E$.

\begin{lemma}\label{sl_2-module}
Let $V$ be a simple $\sl_2$-module.
If $V$ is zad, then $\dim V<\infty$.
\end{lemma}

\begin{proof}
It is straightforward to check that $0\ne H\otimes Fv-F\otimes (H-2)v\in \M_V$.
By the definition of zad modules, we see that $\span\,\K_V=\M_V\neq0$.\smallskip

\noindent{\bf Claim 1.} There exists a nonzero $A\in\sl_2$ such that $A$ acts semisimply on $V$.\smallskip

Choose any nonzero $A\otimes v\in\span\,\K_V, A\in\sl_2, v\in V$. Then $Av=0$.
If $\det(A)\ne 0$, then $A$ is ad-semisimple (which is similar to a multiple of $H$) and acts semisimply on $V$.

If $\det(A)=0$,
then $A$ is ad-nilpotent and thus similar to $E$. Without lost of generality, we may assume that $0\ne E\otimes v\in \span\,\mathcal{K}_V$, or equivalently, $Ev=0$. Since the Casimir element, $(H+1)^2+4FE\in U(\sl_2)$, of $\sl_2$ acts as a scalar on $V$, we see that $$(H+1)^2v=\big((H+1)^2+4FE\big)v\in\C v.$$ Thus, $H$ has an eigenvector in $V$ and hence acts semisimply on $V$, since $V$ is a simple module. 
\smallskip

By Claim 1, we may assume that $H$ acts semisimply on $V$ by replacing $E,F,H$ with $PEP^{-1}, PFP^{-1}, PHP^{-1}$ for some invertible matrix $P$ if necessary, that is, $V$ is a weight module with respect to the Cartan subalgebra $\C H$ of $\sl_2$.\smallskip

\noindent{\bf Claim 2.} $V$ is finite dimensional. 
\smallskip

Suppose on the contrary that $V$ is infinite dimensional. By the representation theory of $\sl_2$,
either $E$ or $F$ acts injectively on $V$.
Without lose of generality, we assume that $E$ acts injectively on $V$.
Take any nonzero $A\otimes v\in \mathcal{K}_V$. Then $Av=0$.
Writing $v$ as a sum of weight vectors and
comparing the highest weights that occur in the expressions of $v$ and $Av$ respectively,
we can easily see $A\in \C H\oplus \C F$.
In particular, $\mathcal{K}_V\subseteq (\C H+\C F)\otimes V$.
On the other hand, we have  $H\otimes Ev-E\otimes (H+2)v\in \M_V=\mathcal{K}_V$ for all nonzero $v\in V$, contradiction. So we must have $\dim V<\infty$.\end{proof}

Now we consider  finite dimensional simple modules over $\sl_2$. Let $m\in\Z_+$ and $V(m)$ be the simple $\sl_2$-module of dimension  $m+1$.
Choose a standard basis $\{v_0,\cdots,v_m\}$ of $V(m)$ such that the actions of $E,F,H$ can be expressed as
$$Fv_i=v_{i+1}, Ev_{i}=i(m+1-i)v_{i-1}, Hv_i=(m-2i)v_i,\ \forall\ i=0,1,\cdots,m,$$
where we have made the convention that $v_i=0$ whenever $i\neq0,1,\cdots,m.$

\smallskip

\begin{lemma}\label{sl_2-even}
Let $V$ be a simple $\sl_2$-module of even dimension.
Then $V$ is zad if and only if  $\dim V=2$.
\end{lemma}

\begin{proof}  Assume that $V=V(m)$ is zad where $m\in\N$ is odd. As in the proof of the previous lemma, we know that $\span\,\K_V=\M_V\neq0$.
Choose any nonzero $A\in\sl_2$ and $v\in V$ such that $A\otimes v\in \mathcal{K}_V$, that is, $Av=0$.
Noticing the fact that any ad-semisimple element acts bijectively on $V$,
we have that $A$ is ad-nilpotent and hence $\det(A)=0$.
If $A\in\C E$, it is clear that $A\otimes v\in \C(E\otimes v_0)$.

Now suppose that $A\not\in \C E$. By replacing $A$ with its nonzero multiple we may assume that   $A=F+\lambda H-\lambda^2E$ for some $\l\in\C$. Assume $v=\sum_{i=0}^mc_iv_i$ for some $c_i\in\C$.
Then from $Av=0$, we compute out that
\begin{equation}\label{action}\aligned
0= & (F+\l H-\l^2E)\sum_{i=0}^mc_iv_i\\
=& \sum_{i=0}^{m-1}c_iv_{i+1}+\l\sum_{i=0}^m(m-2i)c_iv_{i}-\l^2\sum_{i=1}^{m}i(m+1-i)c_iv_{i-1}\\
=& \sum_{i=1}^{m}c_{i-1}v_{i}+\l\sum_{i=0}^m(m-2i)c_iv_{i}-\l^2\sum_{i=0}^{m-1}(i+1)(m-i)c_{i+1}v_{i},
\endaligned\end{equation}
which implies $$\l^2mc_{1}=\l mc_0,\,\,c_{m-1}=\l mc_m$$ and
\begin{equation}\label{c}
c_{i-1}=\l^2(i+1)(m-i)c_{i+1}-\l(m-2i)c_i,\ \forall\ i=1,\cdots, m-1.
\end{equation}
As a result, we deduce $$c_i=\frac{m!}{i!}\l^{m-i}c_m$$ for all $i=0,1,\cdots,m$.

Combining the previous result, we can obtain that
$$\aligned
\span\,\K_V=&\span\,\Big\{E\otimes v_0, \sum_{i=0}^m\frac{m!}{i!}\l^{m-i}(F+\l H-\l^2E)\otimes v_i : \l\in\C\Big\}\\
    = &\span\,\Big\{i(i-1)F\otimes v_{i-2}+iH\otimes v_{i-1}-E\otimes v_i : i=0,1,\cdots,m+2\Big\}.
\endaligned$$
In particular, we have $\dim\span\,\K_V=m+3$.

On the other hand, we know that $\dim \M_V=(\dim \sl_2\otimes V)^2-\dim (\sl_2\cdot V)=2(m+1)$.

If $m\geq 2$, then $\dim\M_V>\dim\span\,\K_V$ and $V$ is not zad. So $m=1$, i.e., $\dim V=2$.

If $V$ is of dimension $2$, i.e., $m=1$, from the above arguments we know that
$\dim\span\,\K_V=4$ and $$\dim\M_V =(\dim \sl_2\otimes V)^2-\dim (\sl_2\cdot V)=4.$$
Thus $\dim\M_V=\dim\span\,\K_V$ and hence $V$ is indeed zad.
\end{proof}

\begin{lemma}\label{sl_2-odd}
If $V$ is a simple $\sl_2$-module of odd dimension, then $V$ is zad.
\end{lemma}

\begin{proof} Suppose $V=V(m)$ where $m=2k>0$ for some $k\in\N$. By a similar argument as in the previous lemma, solving $Av=0$ for $v\in V$ with
$A=F+\l H-\l^2E$ for any $\l\in\C$, we deduce that
$$\mathcal{B}_1=\big\{i(i-1)F\otimes v_{i-2}+iH\otimes v_{i-1}-E\otimes v_i, i=0,1,\cdots,m+2\big\}\subset \span\,\mathcal{K}_V.$$
Recall that we have assumed that $v_i=0$ whenever $i\notin\{0,1,\cdots,m\}$.

Taking $A=H+2\lambda F$ for $\l\in\C$, and solving $Av=0$  for $v\in V$   we deduce that
$$\sum_{i=0}^k\frac{\l^i}{i!}(H+2\l F)\otimes v_{k+i}\in \mathcal{K}_V,  \forall \l\in\C.$$
Thus
$$\mathcal{B}_2=\big\{H\otimes v_{k+i}+2iF\otimes v_{k+i-1}\,|\,i=0,1\ldots,k-1\big\}\subset\span \,\mathcal{K}_V.$$
Note that we did not include the elements for $i=k$ and $i=k+1$, since they have already occurred
in $\mathcal{B}_1$.

Similarly for $A=H-2\lambda E$, by solving $Av=0$ for $v\in V$, we deduce that
$$\sum_{i=0}^k\frac{(k+i)!}{i!(k-i)!}\l^i(H-2\lambda E)\otimes v_{k-i}\in \mathcal{K}_V.$$
Thus
$$\mathcal{B}_3=\big\{(k-i+1)(k+i)H\otimes v_{k-i}-2iE\otimes v_{k-i+1}|i=1,\ldots k-1\big\}\subset\span \,\mathcal{K}_V.$$
Here we do not include the elements for $i=0, i=k$ and $i=k+1$, since they have already occurred
in $\mathcal{B}_1$ or $\mathcal{B}_2$. 

Now it is straightforward to check that $\mathcal{B}_1\cup \mathcal{B}_2\cup \mathcal{B}_3$
is a linearly independent set and has cardinality $2m+2$. We know that
$2m+2\le\dim\span\, \mathcal{K}_V\le \dim \M_V$. Since $\sl_2\cdot V=V$ and $\dim\sl_2\otimes V=3m+3$, we see that
$ \dim\M_V=2m+2$. Thus   $\dim\M_V=\dim\span\,\K_V$ and hence $V$ is  zad.
\end{proof}

Combining Lemma \ref{semidirect} and the above results in this section, we deduce the following

\begin{theorem}\label{gal} 
Suppose $V $ is a simple module over $\sl_2$.  Then the Galilei algebra $\sl_2 \ltimes V$ is zpd if and only if $\dim V =2$ or $\dim V$ is odd.
\end{theorem}

Let $\bb=\C E\oplus\C H$ be the Borel subalgebra of $\sl_2$. Next we consider a subalgebra of the Galilei algebra $\sl_2\ltimes V(m)$. Let  $\bb(m)=\bb\ltimes V(m)\subset \sl_2 \ltimes V(m)$.

\begin{corollary}Let $m\in\N$. \begin{enumerate} 
\item[(a)]  The  $\bb$-module $V(m)$ is zad if and only if $m=2$.
\item[(b)] The Lie algebra $\bb \ltimes V(m)$ is zpd if and only if $m=2$.
\end{enumerate}
\end{corollary}

\begin{proof}  
(a) For any $m\in\N$, we   see that the linear map
$$\bb\otimes V(m)\to V(m), x\otimes v\mapsto xv, \forall x\in \bb, v\in V(m),$$
is onto. Then $\dim\M_{V(m)}=m+1$.
 
 First we assume that $m$ is odd.  It is easy to see that $\K_{V(m)}=\C E\otimes v_0$. Thus $\K_{V(m)}\ne \M_{V(m)}$, i.e.,
$V(m)$ is not zad in this case.

Now we assume that $m=2k$ is even. From the computations in the proof of Lemma 6.3 we see that
$$\mathcal{B}_3=\big\{(k-i+1)(k+i)H\otimes v_{k-i}-2iE\otimes v_{k-i+1}|i=0, 1,\ldots k+1\big\}$$ is a basis of $\mathcal{K}_{V(m)}$, i.e.,
$\dim \K_{V(m)}=k+2$.
Then $V(m)$ is zad if and only if $\M_{V(m)}=\K_{V(m)}$,  if and only if $m+1=k+2$, if and only if $m=2$.

(b) follows from (a) and Lemma 2.12 since only linearly dependent elements in $\bb$ commute.
\end{proof}

From this corollary we obtain a $4$-dimensional non-zpd Lie algebra $\bb(2)=\bb \ltimes V(2)$ where $V(2)$ is the $2$ dimensional module with basis $\{u,v\}$ and action
$$Hu=u, Hv=-v, Ev=u, Eu=0.$$
This $4$-dimensional  Lie algebra $\bb(2)$ is the centerless $1$-spacial aging algebra $\mathfrak{
age}(1)$. See \cite{HS}.

For representations of Galilei algebras, see \cite{LMZ}.

\section{Quantum torus Lie algebras}\label{s7}

Let $q\in\C$, not a root of unity. Let $\C_q[t_1^{\pm1},t_2^{\pm1}]$ be the
associative algebra generated by $t_1^{\pm1}, t_2^{\pm1}$ subject to the relation
$$t_1t_1^{-1}=t_1^{-1}t_1=t_2t_2^{-1}=t_2^{-1}t_2=1,$$
$$ t_2t_1=qt_1t_2.$$
 See \cite{AABGP, EZ}.
Let $\C_q[t_1,t_2]$ be the subalgebra of $\C_q[t_1^{\pm1},t_2^{\pm1}]$ generated by $t_1,t_2$.
For any $\bm=(m_1,m_2)\in\Z^2$, denote ${\bt}^{\bm}=t_1^{m_1}t_2^{m_2}$.
Then it is clear that $${\bt}^{\bm}{\bt}^{\bn}=q^{m_2n_1}\bt^{\bm+\bn},\,\,\,\,{\bt}^{\bm}{\bt}^{\bn}=q^{m_2n_1-n_2m_1}\bt^{\bn}\bt^{\bm}, \forall\bm,\bn\in\Z^2.$$

Let $\LL_q$ be the associated Lie algebra of $\C_q[t_1^{\pm1}, t_2^{\pm1}]$.
 The Lie bracket of $\LL_q$ is given by
$$[{\bt}^{\bm}, {\bt}^{\bn}]=(q^{m_2n_1}-q^{n_2m_1})\bt^{\bm+\bn}.$$
Let $\LL_q^+$ be the associated Lie algebra of $\C_q[t_1,t_2]$.

To show that $\LL_q$ and $\LL^+_q$ are zpd Lie algebras, we need the following  lemma:

\begin{lemma}\label{basis}
For any ${\bn}=(n_1,n_2)\in \Z_+^2$ such that   $n_1,n_2$ are coprime and $n_1n_2\neq0$,
there exists ${\bm}=(m_1,m_2)\in \Z_+^2$ with $(m_1, m_2)\ne (n_1,n_2)$,  $m_1\le n_1 $,   $m_2\le n_2 $
such that  $\{{\bm},{\bn}\}$ forms a $\Z$-basis of $\Z^2$.
\end{lemma}

\begin{proof} Since $n_1,n_2\in\N$ are coprime,
there exists $u,v\in\Z$ such that $un_1-vn_2=1$.
Take $k\in\Z$ such that $0\leq v-kn_1<n_1$ and denote $m_1=v-kn_1, m_2=u-kn_2$.
We have $m_2n_1-m_1n_2=1$. Since $0\leq m_1<n_1$, we also have $0< m_2\leq n_2$.
\end{proof}

\begin{theorem}\label{quantum tori}
If  $q\in\C$  is not a root of unity, then  the Lie algebra $\LL_q$ is zpd.
\end{theorem}

\begin{proof}
Let $f: \LL_q\times\LL_q\rightarrow \C$ be a bilinear map satisfying that
$f(x,y)=0$ whenever $[x,y]=0$. Define a linear map $\Phi: [\LL_q, \LL_q]\rightarrow\C$ as
follows:
$$\aligned
& \Phi({\bt}^{m_1,m_2})=f({\bt}^{(m_1-1,m_2)},{\bt}^{(1,0)})/(q^{m_2}-1),\ \forall\ m_2\neq0;\\
& \Phi({\bt}^{m_1,0})=f({\bt}^{(m_1,-1)},{\bt}^{(0,1)})/(1-q^{m_1}),\ \forall\ m_1\neq0.
\endaligned$$
It is easy to check that
$$\aligned
& f(\bt^{(1,0)}, \bt^{\bm})=\Phi([\bt^{(1,0)}, \bt^{\bm}]),\ \forall\ \bm\in\Z^2,\\
& f(\bt^{(0,1)}, \bt^{(m,-1)})=\Phi([\bt^{(0,1)}, \bt^{(m,-1)}]),\ \forall\ m\in\Z.
\endaligned$$
Define $$\phi(x,y)=f(x,y)-\Phi([x,y])$$ for all $x, y\in\LL_q$ and denote
$$T=\big\{\bm\in\Z^2\ |\ \phi(\bt^{\bm}, \bt^{\bn})=0,\ \forall\ \bn\in\Z^2\big\}.$$
We have $(0,0), (1,0)\in T$. To show that $\LL_q$ is zpd, we need only to show that $T=\Z^2$.
\smallskip

\noindent{\bf Claim 1.}
If $\bm\in T$, then $j\bm\in T$ for all $j\in\Z_+$ and
\begin{equation}\label{bna}\aligned
& \phi(\bt^{\bn}, \bt^{j\bm+k\bn})=0,\ \forall\ \bn\in\Z^2, k,j\in\Z_+.
\endaligned\end{equation}
\smallskip

For any $\l\in\C, k\in\N$ and $\bn\in\Z^2$, we always have
$[\l\bt^{\bm}+\bt^{\bn}, (\l\bt^{\bm}+\bt^{\bn})^k]=0$,
where
$$(\l\bt^{\bm}+\bt^{\bn})^k=\sum_{i=0}^k\l^ia_i(q)\bt^{i\bm+(k-i)\bn}$$
and each $a_i(q)$ is a polynomial function in $q$.

By the assumption on $f$, we have
$$f(\l\bt^{\bm}+\bt^{\bn}, \sum_{i=0}^k\l^ia_i(q)\bt^{i\bm+(k-i)\bn})=0
=\Phi([\l\bt^{\bm}+\bt^{\bn}, \sum_{i=0}^k\l^ia_i(q)\bt^{i\bm+(k-i)\bn}]),\ \forall\ \l\in\C.$$
In particular, $$\phi(\l\bt^{\bm}+\bt^{\bn}, \sum_{i=0}^k\l^ia_i(q)\bt^{i\bm+(k-i)\bn})=0, \,\,\,\forall  \l\in\C.$$ Since $\bt^{\bm}\in T$, we see $$\sum_{i=0}^k\l^ia_i(q)\phi(\bt^{\bn}, \bt^{i\bm+(k-i)\bn})=0,\,\,\,\forall \l\in\C,$$ which implies
$$a_i(q)\phi(\bt^{\bn}, \bt^{i\bm+(k-i)\bn})=0,\ \forall\ \bn\in\Z^2, i=0,1,\cdots,k.$$
Taking $i=k$, we deduce $\phi(\bt^{\bn}, \bt^{km})=0$ forcing $k\bm\in T$
since $a_{k}(q)=q^{\frac{k(k-1)m_1m_2}{2}}\neq0$.
Taking $i=1$, we deduce $\phi(\bt^{\bn}, \bt^{\bm+(k-1)\bn})=0$ for all $k\in\N$
since $a_1(q)$ is a nonzero multiple of $(1+q_{\bm,\bn}+\cdots+q^{k-1}_{\bm,\bn})$ and hence nonzero,
where $q_{\bm,\bn}$ is defined by $\bt^{\bm}\bt^{\bn}=q_{\bm,\bn}\bt^{\bn}\bt^{\bm}$.
Replacing $\bm$ with $j\bm$ for $j\in\Z_+$ in the above argument, we get the result of this claim.
\smallskip

Now take arbitrary $\bm, \bn\in\Z^2$. For any $\l, \mu\in\C$ and $k\in\N$,
we consider the identity
$[(\l\bt^{\bm}+\mu\bt^{-\bm}+\bt^{\bn}, (\l\bt^{\bm}+\mu\bt^{-\bm}+\bt^{\bn})^k]=0$.
Suppose
$$(\l\bt^{\bm}+\mu\bt^{-\bm}+\bt^{\bn})^k=\sum_{i,j=0}^k\l^i\mu^ja_{ij}(q)\bt^{i\bm-j\bm+(k-i-j)\bn},$$
where $a_{ij}$ are some polynomials in $q$. By a similar argument as before, we can deduce 
\begin{equation}\label{lmu}
\phi\Big(\l\bt^{\bm}+\mu\bt^{-\bm}+\bt^{\bn}, \sum_{i,j=0}^k\l^i\mu^ja_{ij}(q)\bt^{(i-j)\bm+(k-i-j)\bn}\Big)=0,\ \forall\ \l, \mu\in\C.
\end{equation} \smallskip

\noindent{\bf Claim 2.} If $\bm\in T$, then $j\bm\in T$ for all $j\in\Z$ and
$$\phi(\bt^{\bn}, \bt^{j\bm+k\bn})=0,\ \forall\ \bn\in\Z^2, j\in\Z, k\in\Z_+.$$
\smallskip

Taking  $\bm\in T$ in \eqref{lmu}, we see
\begin{equation}\label{mu}
\sum_{i,j=0}^k\l^i\mu^ja_{ij}(q)\phi(\mu\bt^{-\bm}+\bt^{\bn}, \bt^{(i-j)\bm+(k-i-j)\bn})=0,\
 \forall\ \l, \mu\in\C.
\end{equation}
Regarding the left hand side of \eqref{mu} as a polynomial in $\l, \mu$ and considering the coefficient
of $\l^{k-1}\mu$, we see that
$$a_{k-1,1}(q)\phi(\bt^{\bn}, \bt^{(k-2)\bm})+a_{k-1,0}(q)\phi(\bt^{-\bm}, \bt^{(k-1)\bm+\bn})=0.$$
Taking $k=2$ and noticing $a_{1,0}(q)=q^{m_2n_1}+q^{n_2m_1}\neq0$,
we see that $\phi(\bt^{-\bm}, \bt^{\bm+\bn})=0$ for all $\bn\in\Z^2$,  that is, $-\bm\in T$.
Now by Claim 1, we get $j\bm\in T$ and $\phi(\bn, j\bm+k\bn)=0$ for all $j\in\Z, k\in\Z_+$.
\smallskip

\noindent{\bf Claim 3.} If $\bm\in T$, then 
$\phi(\bt^{\bn}, \bt^{j\bm+k\bn})=0,\ \forall\ \bn\in\Z^2, j\in\Z, k\in\Z, k\neq-1.$
\smallskip

Exchanging $\bm$ and $\bn$ in \eqref{lmu}, we get
\begin{equation*}
\phi\Big(\l\bt^{\bn}+\mu\bt^{-\bn}+\bt^{\bm}, \sum_{i,j=0}^k\l^i\mu^ja_{ij}(q)\bt^{(i-j)\bn+(k-i-j)\bm}\Big)=0,\ \forall\ \l, \mu\in\C.
\end{equation*}
Note that $a_{ij}(q)$ has been changed since we have exchanged $\bm$ and $\bn$.

Since $\bm\in T$, we have
\begin{equation}\label{mul}
\sum_{i,j=0}^k\l^i\mu^ja_{ij}(q)\phi\Big(\l\bt^{\bn}+\mu\bt^{-\bn}, \bt^{(i-j)\bn+(k-i-j)\bm}\Big)=0,\ \forall\ \l, \mu\in\C.
\end{equation}
Considering the coefficient of $\l\mu^{k-1}$, we get
$$a_{0,k-1}(q)\phi(\bt^{\bn}, \bt^{(1-k)\bn+\bm})+a_{1,k-2}(q)\phi(\bt^{-\bn}, \bt^{(3-k)\bn+\bm})=0,\ \forall\ k\in\N.$$
If $k\geq 3$, we already have $\phi(\bt^{-\bn}, \bt^{(3-k)\bn+\bm})=0$ by Claim 2 and
hence $$a_{0,k-1}(q)\phi(\bt^{\bn}, \bt^{(1-k)\bn+\bm})=0.$$
It is straightforward to calculate that $a_{0,k-1}(q)$ is a nonzero multiple of
$1+q^{-1}_{\bm,\bn}+\cdots+q^{1-k}_{\bm,\bn}$ and hence nonzero. We see that
$\phi(\bt^{\bn}, \bt^{(1-k)\bn+\bm})=0$ for $k\geq3$.
Replacing $\bm$ with $j\bm, j\in\Z$ in the above arguments, we see
$\phi(\bt^{\bn}, \bt^{-k\bn+j\bm})=0$ for all $k,j\in\Z$ with $k\geq2$.
The claim follows from Claim 2. 
\smallskip

\noindent{\bf Claim 4.} $T=\Z^2$. 
\smallskip

Recalling that $(1,0)\in T$, we have $\phi(\bt^{(0,1)},\bt^{(j,k)})=0$ for $j,k\in\Z, k\neq-1$ by Claim 3.
On the other hand, by the definition of $\Phi$,
we have $f(\bt^{(0,1)}, \bt^{(j,-1)})=\Phi([\bt^{(0,1)}, \bt^{(j,-1)}])$ for all $j\in\Z$
and hence $\phi(\bt^{(0,1)}, \bt^{(j,-1)})=0$. As a result, $(0,1)\in T$.
By Claim 2, we also have $(0,k), (k, 0)\in T$ for all $k\in\Z$.

Now take any $\bn=(n_1,n_2)\in\Z^2$ such that $n_1,n_2$ are coprime.
We will prove $\bn\in T$ by induction on $|n_1n_2|$.
For $|n_1n_2|=0$, the result is true by the previous discussion.

Now suppose $|n_1n_2|\geq1$. Without loss of generality we may assume that $\bn\in\Z_+^2$. By Lemma \ref{basis}, there exists $\bm=(m_1,m_2)\in\Z_+^2$ with $(m_1, m_2)\ne (n_1,n_2)$,  $m_1\le n_1 $,   $m_2\le n_2 $,
such that $\{\bm,\bn\}$ forms a $\Z$-basis of $\Z^2$.
Denote $\bm'=(n_1-m_1, n_2-m_2)\in\Z_+^2$.
Then $\{\bm',\bn\}$ forms a basis of $\Z^2$ and $|(n_1-m_1)(n_2-m_2)|<|n_1n_2|$.
By induction hypothesis, we have $\bm,\bm'\in T$. Then Claim 3 indicates
$$\phi(\bt^{\bn}, \bt^{j\bm+k\bn})=0,\ \phi(\bt^{\bn}, \bt^{j\bm'+k\bn})=0,\ \forall\ j,k\in\Z, k\neq-1.$$
For any $\bn'\in\Z^2$, we can write $\bn'=j_1\bm+k_1\bn=j_2\bm'+k_2\bn$ for some $j_1,k_1,j_2,k_2\in\Z$. Noticing that $k_1$ and $k_2$ can not both be equal to $-1$, we have $\phi(\bt^{\bn},\bt^{\bn'})=0$ and $\bn\in T$. By Claim 2, we see $k\bn\in T$ for all $k\in\Z$. The claim and hence the theorem follows.
\end{proof}

Using similar methods as in the previous theorem, we can also show that $\LL_q^+$ is a zpd Lie algebra.

\begin{theorem}\label{quantum plane}
If  $q\in\C$  is not a root of unity, then  the Lie algebra $\LL^+_q$ is zpd.
\end{theorem}

\begin{proof} Note that $ [\LL^+_q,\LL^+_q]=$span$\{t_1^it_2^j:i,j\in\N\}$.
Let $f: \LL^+_q\times\LL^+_q\rightarrow \C$ be a bilinear map satisfying that
$f(x,y)=0$ whenever $[x,y]=0$. Define a linear map $\Phi: [\LL^+_q,\LL^+_q]\rightarrow\C$ as
follows:
$$\aligned
& \Phi({\bt}^{(m_1,m_2)})=f({\bt}^{(m_1-1,m_2)},{\bt}^{(1,0)})/(q^{m_2}-1),\ \forall\ m_1, m_2\in\N.\\
\endaligned$$
For convenience, we denote $$\phi(x,y)=f(x, y)-\Phi([x, y])$$ for all $x,y\in\LL_q^+$.
Denote $$T=\big\{\bm\in\Z_+^2\ |\ \phi(\bt^{\bm}, \bt^{\bn})=0,\ \forall\ \bn\in\Z_+^2\big\}.$$
To show that $\LL^+_q$ is zpd, we need only to show $\phi(\bm,\bn)=0$ for all $\bm, \bn\in\Z_+^2$,
or equivalently, $T=\Z_+^2$.
Note that $(0,0), (1,0)\in T$. \smallskip
\smallskip 

\noindent{\bf Claim 1.} If $\bm\in T$, then $k\bm\in T$ for all $k\in\Z_+$ and
\begin{equation}\label{bn}\aligned
& \phi(\bt^{\bn}, \bt^{j\bm+k\bn})=0,\ \forall\ \bn\in\Z_+^2, j,k\in\Z_+.
\endaligned\end{equation}
\smallskip

For any $\l\in\C, k\in\N$ and $\bn\in\Z_+^2$, we always have
$[\l\bt^{\bm}+\bt^{\bn}, (\l\bt^{\bm}+\bt^{\bn})^k]=0$,
where
$$(\l\bt^{\bm}+\bt^{\bn})^k=\sum_{i=0}^k\l^ia_i(q)\bt^{i\bm+(k-i)\bn}$$
and each $a_i(q)$ is a function on $q$. 

By the assumption on $f$, we have
$$f(\l\bt^{\bm}+\bt^{\bn}, \sum_{i=0}^k\l^ia_i(q)\bt^{i\bm+(k-i)\bn})=0
=\Phi([\l\bt^{\bm}+\bt^{\bn}, \sum_{i=0}^k\l^ia_i(q)\bt^{i\bm+(k-i)\bn}]),\ \forall\ \l\in\C.$$
In particular, $$\phi(\l\bt^{\bm}+\bt^{\bn}, \sum_{i=0}^k\l^ia_i(q)\bt^{i\bm+(k-i)\bn})=0$$ for all $\l\in\C$. Since $\bm\in T$, we see 
$$\sum_{i=0}^k\l^ia_i(q)\phi(\bt^{\bn}, \bt^{i\bm+(k-i)\bn})=0$$ for all $\l\in\C$, which implies
$$a_i(q)\phi(\bt^{\bn}, \bt^{i\bm+(k-i)\bn})=0,\ \forall\ i=0,1,\cdots,k.$$
Taking $i=k$, we deduce $\phi(\bt^{\bn}, \bt^{km})=0$ forcing $k\bm\in T$
since $a_{k}(q)\neq0$.
Taking $i=1$, we deduce $\phi(\bt^{\bn}, \bt^{\bm+(k-1)\bn})=0$ for all $k\in\N$ since $a_1(q)\neq0$.
Replacing $\bm$ with $j\bm$ for $j\in\Z_+$ in the above argument, we get the result of this claim.
\smallskip

\noindent{\bf Claim 2.} $T=\Z_+^2$. 
\smallskip

By Claim 1, we see $(0,1)\in T$ by the fact that $(1,0)\in T$. Hence $(j,0), (0,k)\in T$ for all $j,k\in\Z_+$. 
Now for any $\bn\in\Z_+^2$ and $\l,\mu\in\C$, the identity 
$$[\bt^{\bn}+\l\bt^{(j,0)}+\mu\bt^{(0,k)}, (\bt^{\bn}+\l\bt^{(j,0)}+\mu\bt^{(0,k)})^2]=0,\ \forall\ j,k\in\Z_+$$
implies 
$$\phi\big(\bt^{\bn}, \l(\bt^{\bn}\bt^{(j,0)}+\bt^{(j,0)}\bt^{\bn})+\mu(\bt^{\bn}\bt^{(0,k)}+\bt^{(0,k)}\bt^{\bn})+\l\mu(\bt^{(j,0)}\bt^{(0,k)}+\bt^{(0,k)}\bt^{(j,0)})\big)=0,$$
forcing $\phi(\bt^{\bn},\bt^{(j,k)})=0$ for all $j,k\in\Z_+$. The result follows. 
\end{proof}

\section{Affine Lie algebras}\label{s8}
In this section we assume that $\F=\C$.
Let ${\g}$ be a finite-dimensional simple Lie algebra  with Cartan decomposition $\g=\h\oplus \oplus_{\a\in \Delta}\g_{\a}$, where  $\Delta$ is root system of $\g$ with repect to a Cartan subalgebra $\h$ and $$\g_{\a}=\{x\in \g|[h,x]=\a(h)x,\forall h\in \h\}.$$  We write $(x,y)$ for the Killing form
 on ${\g}$.

Let $\Delta^\vee=\{h_\alpha|\alpha\in\Delta\}$ be the coroot system of $\Delta$, and $\{e_\alpha|\alpha\in\Delta\}$ be the Chevalley basis of $\g$ (See Chapter 7 in \cite{H}). For any $\alpha\in\Delta$ we know that
$$\aligned &h_{-\alpha}=-h_\alpha,\,\,\, \alpha(h_\alpha)=2,\\
&[e_{\a},e_{-\a}]=h_{\a}, \,\,\,[h_{\a},e_{\pm \a}]=\pm 2e_{\pm \a}.\endaligned$$
Since the Killing form is invariant, we have
$$(h_{\a},h_{\a})=([e_{\a},e_{-\a}],h_{\a})=(e_{\a},[e_{-\a},h_{\a}])=2(e_{\a},e_{-\a}),\forall\a\in\Delta.$$ 

 The (untwisted) affine Lie algebra  associated
with ${\g}$ is defined as $$ {\tg}= {\g} \otimes {\C}[t,t^{-1}] \oplus
{\C}c \oplus {\C}d, $$
where $c$ is the canonical central element
 and  the Lie algebra structure
is given by (see \cite{K})
\begin{equation*} 
 [ x \otimes t^n, y \otimes t^m] = [x,y] \otimes t^{n+m} + n (x,y) \delta_{n+m,0} c, \end{equation*}
$$
[d, x \otimes t^n] = n x \otimes t^n $$
for $x,y \in {\g}$ and $m,n\in \mathbb{Z}$.  We will
write $x(n)$ for $x \otimes t^{n}$, and
$$\hg=[\tg,\tg]=\g\otimes \C[t,t^{-1}]+\C z,$$
 $$\hat\h=\h\otimes \C[t,t^{-1}]+\C c.$$
  We first prove the following crucial formulae.

\begin{lemma}\label{lemma-aff} For any $\alpha\in\Delta$,  $h\in\h$ and $i, j\in\Z$  we have 
\begin{enumerate}
\item[(a)]
 $h(i)\otimes e_{\a}(j)\equiv h(0)\otimes e_{\a}(i+j)\,\,\big({\rm mod} \,\,{\rm span}\, \mathcal{K}_{\hg}\big);$

\item[(b)] $e_{\a}(i)\otimes e_{-\a}(-i)\equiv e_{\a}(0)\otimes e_{-\a}(0)+\frac{1}{2}
h_{\a}(i)\otimes h_{\a}(-i) \,\,\big({\rm mod} \,\,{\rm span}\, \mathcal{K}_{\hg}\big);$

\item[(c)] $e_{\a}(i)\otimes e_{\beta}(j)\equiv  e_{\a}(0)\otimes e_{\beta}(i+j)\,\,\big({\rm mod} \,\,{\rm span}\, \mathcal{K}_{\hg}\big)$, if  $ (\a+\beta,i+j) \ne (0,0)$;

\item[(d)] $d\otimes e_{\a}(i)\equiv \frac{i}{2} h_{\a}(0)\otimes e_{\a}(i) \,\,\big({\rm mod} \,\,{\rm span}\, \mathcal{K}_{\tilde{g}}\big);$

 \item[(e)] $d\otimes h_{\a}(i)\equiv ie_{\a}(0)\otimes e_{-\a}(i) \,\,\big({\rm mod} \,\,{\rm span}\, \mathcal{K}_{\tilde{g}}\big).$
\end{enumerate}
\end{lemma}

\begin{proof} (a)  
From $$[h(i)+e_{\a}(i+j), h(0)+e_{\a}(j)]=0$$ we see that $$(h(i)+e_{\a}(i+j))\otimes (h(0)+e_{\a}(j))\in  \mathcal{K}_{\hg}.$$ Hence we obtain (a).

(b) From $$[e_{-\a}(0)+e_{\a}(i)+\frac{\sqrt{2}}{2}(h_{\a}(0)-h_{\a}(i)),e_{-\a}(-i)+e_{\a}(0)+\frac{\sqrt{2}}{2}(h_{\a}(-i)-h_{\a}(0))]=0$$ we see that  $$(e_{-\a}(0)+e_{\a}(i)+\frac{\sqrt{2}}{2}(h_{\a}(0)-h_{\a}(i)))\otimes (e_{-\a}(-i)+e_{\a}(0)+\frac{\sqrt{2}}{2}(h_{\a}(-i)-h_{\a}(0)))\in  \mathcal{K}_{\hg}.$$
 Using (a) we get (b). 
 
 Further, (c) follows from
 $$(e_{\a}(i)+e_{\beta}(i+j))\otimes (e_{\beta}(j)+e_{\a}(0))\in \mathcal{K}_{\hat{g}},$$
and (d) follows from
$$(d-\frac{i}{2} h_{\a}(0))\otimes  e_{\a}(i)\in  \mathcal{K}_{\tilde{g}}.$$
Finally, applying (a),(d) to $$(d+\frac{i}{2}h_{\a}(0)-ie_{\a}(0))\otimes (h_{\a}(i)-e_{\a}(i)+e_{-\a}(i))\in  \mathcal{K}_{\tilde{g}}$$ we obtain (e).
\end{proof}

\begin{theorem}The untwisted affine Lie algebras $\tg$ and $\hg$ are zpd. \end{theorem}

\begin{proof}  We will first prove that that $\hg$ is zpd. Note that $c\otimes {\hg}\in \mathcal{K}_{\hg}$. Let $X\in \mathcal{M}_{\hg}$. 
Using Lemma \ref{lemma-aff} (a)-(c), we see that $$X\in (\g\otimes 1)\otimes \hg+\hat\h\otimes \hat\h +\span\,\, \mathcal{K}_{\hat{\g}}.$$ We may assume that $X\equiv X_1+X_1' \mod \span(\K_{\hg})$ where $X_1=\sum_{i=1}^ru_i\otimes v_i\in (\g\otimes 1)\otimes \hg$ and $X_1'=\sum_{i=1}^su'_i\otimes v'_i\in \hat\h\otimes \hat\h$. From $$\sum_{i=1}^r[u_i, v_i]+\sum_{i=1}^s[u'_i, v'_i]=0,$$
and $[\g\otimes 1),\hg]\cap [\hat\h,\hat\h]=0$ we deduce that
$$\sum_{i=1}^r[u_i, v_i]=0,\,\,\,\sum_{i=1}^s[u'_i, v'_i]=0,$$
i.e., $X_1\in ((\g\otimes 1)\otimes \hg)\cap \mathcal{M}_{\hg}$ and $X_1'\in \mathcal{M}_{\hat\h}$. 
Hence
$$\mathcal{M}_{\hg}=((\g\otimes 1)\otimes \hg+\hat\h\otimes \hat\h +\span\,\, \mathcal{K}_{\hat{\g}})\cap \mathcal{M}_{\hg}=(\g\otimes 1)\otimes \hg)\cap \mathcal{M}_{\hg}+\mathcal{M}_{\hat\h} +\span\,\, \mathcal{K}_{\hat{g}}.$$
 However, $\hat\h$ is zpd and $\hat{\g}$ is a zad $\g$-module. Thus $$((\g\otimes 1)\otimes \hg)\cap \mathcal{M}_{\hg}\subseteq \span\,\, \mathcal{K}_{\hat{\g}}$$ and
$$\mathcal{M}_{\hat\h}=\span\,\, \mathcal{K}_{\hat{\h}}\subseteq \span\,\, \mathcal{K}_{\hg}.$$ Therefore we have $\mathcal{M}_{\hg}=\span\,\, \mathcal{K}_{\hg}$, as desired.

Next we prove that $\tg$ is zpd. Let $Y\in \mathcal{M}_{\tg}$. In view of Lemma \ref{lemma-aff} (d) and (e), using the fact $c\otimes {\tg}\in \mathcal{K}_{\tg}$ we see that $$Y\in \hg\otimes \hg+   \span\mathcal{K}_{\tg}.$$ Noting that  $ (\hg\otimes \hg)\cap  \mathcal{M}_{\tg}= \mathcal{M}_{\hg}$, and using the established result that $\mathcal{M}_{\hg}= \span\mathcal{K}_{\hg}$ we see that 
$$Y\in  (\hg\otimes \hg)\cap  \mathcal{M}_{\tg}+   \span\mathcal{K}_{\tg}=\mathcal{M}_{\hg}+   \span\mathcal{K}_{\tg}\subset  \span \mathcal{K}_{\hg}+  \span \mathcal{K}_{\tg}=  \span\mathcal{K}_{\tg}.$$
Thus $\mathcal{M}_{\tg}= \span\mathcal{K}_{\tg}$, i.e., $\tg$ is zpd.
 \end{proof}

\section{Commutativity preserving maps}\label{s9}

Let $\varphi$ be a linear map from one algebra  into another.
We say that  $\varphi$  {\em preserves   commutativity} if $\varphi(x)$ and $\varphi(y)$ commute whenever $x$ and $y$ commute. The problem of describing such maps has a long  history in the context of associative algebras. It originated in linear algebra and operator theory, and later, in connection with the development of the theory of functional identities, moved to noncommutative ring theory. We refer the reader to \cite[pp. 218-219]{FI} for historic details and references. One usually assumes that $\varphi$ is bijective or at least surjective. In this case a natural possibility is that $\varphi$ is a linear combination of a homomorphism or an antihomomorphism and a map having the range in the center of the target algebra. Without the surjectivity assumption the problem becomes much more involved as we have another natural possibility, i.e., maps with commutative range. The concept of a zpd Lie algebra  actually arose from
the problem  of describing not necessarily surjective commutativity preserving linear maps on finite dimensional central simple algebras \cite{BS}. The solution was based on
first showing that
 $M_n(\F)$, viewed as a Lie algebra, is zpd.

Although the notion of a commutativity preserving linear map is Lie-theoretic in nature,  to the best of our knowledge
it was studied by Lie algebra tools
  only in a few papers \cite{Wang, WZ, Wong}. Our aim now is to show how it can be handled in zpd Lie algebras. More precisely, we will assume that  the first Lie algebra is zpd and the target Lie algebra satisfies conditions suitable for using the theory of functional identities. Indeed functional identites are a standard tool in treating
 commutativity preservers, but	results of such a type are nevertheless new. In light of examples of zpd Lie algebras that have been found in previous sections, they  are now applicable to various concrete situations.

Let us begin with a simple example indicating the delicacy of the problem.

\begin{example}\label{bex}
Let $L$ be a Lie algebra in which only  linearly dependent elements commute. Then {\em every} linear map $\varphi:L\to L$ preserves commutativity. Indeed such Lie algebras usually are not zpd, but some of them, like  $\sl_2(\F)$, are. This shows that  one cannot expect that commutativity preserving maps between two zpd Lie algebras can be always nicely   described.  Some other conditions are needed, too.
\end{example}

The condition that we will require is that the target Lie algebra is a {\em $3$-free subset} of an associative unital algebra.  To avoid making this section too lengthy,
we only refer to the book \cite{FI} for the definition and basic properties of such sets, and give the proof of the next lemma without a detailed explanation of the results that are used.

\begin{lemma}\label{lkl}
Let $V$ be a vector space over a field $\F$ with {\rm char}$(\F)\ne 2$, let $A$ be a unital associative algebra over $\F$, let $\varphi:V\to A$ be a linear map and $B:V\times V\to A$ be a skew-symmetric bilinear map. Suppose that
\begin{equation}\label{jk}[\varphi(x),B(y,z)] + [\varphi(z),B(x,y)] + [\varphi(y),B(z,x)] =0,\ \forall\ x,y,z\in V.\end{equation}
If the range of $\varphi$ is a $3$-free subset of $A$, then there exist $\lambda \in Z$, the center of $A$, and a skew-symmetric bilinear map $\nu:V\times V\to Z$ such that
$$B(x,y)=\lambda [\varphi(x),\varphi(y)]+\nu(x,y),\ \forall\ x,y\in V.$$
\end{lemma}

\begin{proof}
Since the range of $\varphi$ is $3$-free it follows from \cite[Theorem 4.13]{FI} that there exist elements $\lambda,\lambda'\in Z$ and maps $\mu,\mu':V\to Z$, $\nu:V\times V\to Z$
such that
$$B(x,y)= \lambda\varphi(x)\varphi(y)+\lambda'\varphi(y)\varphi(x) + \mu(x)\varphi(y)+\mu'(y)\varphi(x) + \nu(x,y),\ \forall\ x,y\in V. $$
As $B$ is bilinear,  \cite[Lemma 4.6, Remark 4.7]{FI} imply that $\mu,\mu'$ are linear maps and $\nu$ is bilinear. Further, since
$B(x,y)=-B(y,x)$ holds by assumption, it follows that
$$(\lambda + \lambda')\varphi(x)\varphi(y) + (\lambda + \lambda')\varphi(y)\varphi(x)   +(\mu + \mu')(x)\varphi(y) +  (\mu + \mu')(y)\varphi(x) + \nu(x,y) + \nu(y,x)=0.$$
Now, \cite[Lemma 4.4]{FI} implies that $\lambda=-\lambda'$, $\mu=-\mu'$, and $\nu$ is skew-symmetric. Thus, we have
$$B(x,y)= \lambda[\varphi(x),\varphi(y)] + \mu(x)\varphi(y)-\mu(y)\varphi(x)+\nu(x,y).$$
Using this form in \eqref{jk} we obtain
$$2\mu(y)[\varphi(x),\varphi(z)] + 2\mu(z)[\varphi(y),\varphi(x)] + 2\mu(x)[\varphi(z),\varphi(y)]=0.$$
Since char$(\F)\ne 2$ it follows from \cite[Lemma 4.4]{FI} that $\mu=0$.
\end{proof}

We will now derive two theorems on commutativity preservers from this lemma. To understand their meaning, it is important to mention  that Lie ideals of associative algebras are their $3$-free subsets under rather mild assumptions \cite[Corollary 5.16]{FI}. The simplest, but very illustrative example is that $\sl_n(\F)$ is a $3$-free subset of $M_n(\F)$ provided that  $n \ge 3$. The $n=2$ case is an exception and indeed the next theorem does not hold if $L=L'= \sl_2(\F)$, cf. Example \ref{bex}. Similarly,
 Lie algebras of skew-symmetric elements in associative algebras with involution  are ``usually" $3$-free subsets \cite[Corollary 5.18]{FI}, and so are  their Lie ideals \cite[Corollary 5.19]{FI}. It seems reasonable to conjecture that there are many other types of Lie algebras that are $3$-free subsets of some of their associative envelopes. However, only little is known  about this at present.

\begin{theorem}
Let $L$  and  $L'$ be a Lie algebras over a field $\F$ with {\rm char}$(\F)\ne 2$. Suppose that  $L$ is zpd, $L'$ is centerless, and  $L'$ is Lie subalgebra of an associative unital algebra $A$ whose center is equal to $\F1$.
If $L'$ is a $3$-free subset of $A$, then
 every commutativity preserving bijective linear map $\varphi:L\to L'$ is a scalar multiple of an isomorphism.
\end{theorem}

\begin{proof} The condition that $\varphi$ preserves commutativity can be interpreted as that the map $f(x,y)= [\varphi(x),\varphi(y)]$ satisfies the condition
from the definition of the zpd property. Thus, since
 $L$ is zpd it follows (from Lemma \ref{lx})  that there exists a linear map $\Phi:[L,L]\to L'$ such that
\begin{equation}\label{foi}[\varphi(x),\varphi(y)] = \Phi([x,y]),\ \forall\ x,y\in L.\end{equation}
Applying the Jacobi identity this yields
$$[\varphi(x),\varphi([y,z])] + [\varphi(z),\varphi([x,y])] + [\varphi(y),\varphi([z,x])] =0,\ \forall\ x,y,z\in L.$$
This makes it possible for us to apply Lemma \ref{lkl}. Accordingly,
there exist  $\lambda\in \F$ and a skew-symmetric bilinear map $\nu:L\times L\to \F 1$ such that
\begin{equation}\label{fd}\varphi([x,y])= \lambda[\varphi(x),\varphi(y)]  + \nu(x,y),\ \forall\ x,y\in L.\end{equation}
However, since the center of $L'$ is $0$ by assumption it follows from \eqref{fd} that $\nu=0$. Consequently, $\lambda\varphi$ is a Lie algebra homomorphism. We only have to show that $\lambda\ne 0$. Assume, therefore, that $\lambda =0$. Then $\varphi([x,y])=0$ for all $x,y\in L$, yielding that $L$ is Abelian. Since $\varphi$ preserves commutativity this implies that $L'$ is Abelian, too. However, this is impossible for $L'$ is a $3$-free subset of $A$.
\end{proof}

The next, and the last theorem is similar. The difference is that we impose another conditions on $L$ and less conditions on $\varphi$, $
A$, and $L'$ -- in fact, $L'$ does not even appear
 since we do not  need to assume that the range of $\varphi$ is a Lie subalgebra of $A$.

\begin{theorem}\label{tat}
Let  $L$ be a  perfect zpd Lie algebra over a field $\F$ with {\rm char}$(\F)\ne 2$, and let $A$ be an associative unital algebra over $\F$ whose center $Z$ is a field. If
 $\varphi:L\to A$ is a commutativity preserving linear map whose range is a $3$-free subset of $A$, then $\varphi$ is of the form \begin{equation}\label{f}\varphi(x)=\alpha\theta(x) + \beta(x)\end{equation} where $\alpha\in Z$, $\theta$ is a Lie algebra homomorphism from $L$ into $A$ and
$\beta$ is a linear map from $L$ into $Z$.
\end{theorem}

\begin{proof}
Just as in the previous proof one derives that   \eqref{foi} holds for some
linear map $\Phi:[L,L]\to A$ and
\eqref{fd} holds for some $ \lambda\in Z$ and a skew-symmetric bilinear map $\nu:L\times L\to Z$. We remark that the range of $\varphi$ cannot be contained in $Z$ for it is a $3$-free subset of $A$. Since $L$ is perfect, this implies that $\lambda\ne 0$. Define $\theta,\beta:L\to A$ by
$$\theta(x)= \lambda^2\Phi(x)$$
and
$$ \beta(x)= \varphi(x) - \lambda^{-1}\theta(x).$$
Writing $\alpha$ for $\lambda^{-1}$ we see that all that remains to show is that $\beta$ maps into $Z$ and that
$\theta$ is a Lie algebra homomorphism. We have
$$
\beta([x,y])=  \varphi([x,y]) - \lambda^{-1}\theta([x,y])=  \lambda[\varphi(x),\varphi(y)]  + \nu(x,y) - \lambda\Phi([x,y])=\nu(x,y)\in Z.
$$
Since $L$ is perfect this proves that $\beta$ maps into  in $Z$. Consequently,
\begin{eqnarray*}
\theta([x,y])&=& \lambda^2\Phi([x,y])= \lambda^2[\varphi(x),\varphi(y)]\\
&=&[\lambda\varphi(x),\lambda\varphi(y)] = [\theta(x) + \lambda\beta(x),\theta(y) + \lambda\beta(y)] = [\theta(x) ,\theta(y)],
\end{eqnarray*}
proving that $\theta$ is a homomorphism.
\end{proof}

It should be mentioned that we now know  quite a few  perfect zpd Lie algebras, for example, complex finite-dimensional semisimple  Lie algebras,
 loop algebras, untwisted affine Lie algebras, quantum torus Lie algebras in Sect.7, and some Gallilei algebras. Theorem \ref{tat} therefore 
generalizes some results from \cite{Wang, WZ, Wong}. 

\bigskip

{\bf Acknowledgement.} The majority of this paper was conducted during the four authors' visit to Soochow University, China, in the August of 2016. M. Bre\v sar is partially supported by ARRS grant P1--0288; X.Guo is partially supported by NSF of China (Grant 11101380) and the Outstanding Young Talent Research Fund of Zhengzhou University (Grant 1421315071); G. Liu is partially supported by NSF of China (Grant 11301143) and the school fund of Henan University (2012YBZR031, yqpy 20140044); R.L\" u is partially supported by NSF of China (Grant 11471233, 11371134); K. Zhao is partially supported by NSF of China (Grants 11271109, 11471233) and NSERC.

\end{document}